\documentclass[a4paper,11pt]{article}

\usepackage[small,bf,hang]{caption} 

\usepackage{todonotes}
\usepackage[ruled,linesnumbered]{algorithm2e}
\usepackage{amsmath,amsthm,amssymb}
\usepackage{graphicx,subcaption,tikz}
\usepackage{comment,xspace,paralist}							
\usepackage[shortlabels]{enumitem}
\setlist[itemize]{itemsep=2pt,parsep=0pt,topsep=4pt,partopsep=0pt}
\setlist[enumerate]{label=\upshape(\alph*),itemsep=2pt,parsep=0pt,topsep=4pt,partopsep=0pt}	
\usepackage{hyperref}
\usepackage{changepage}
\hypersetup{
    colorlinks=true,
    citecolor = blue,
    linkcolor=blue,
    filecolor=magenta,
    urlcolor=cyan,
}
\usepackage[margin=1in]{geometry}
\usepackage{verbatim}
\usepackage[capitalize]{cleveref}
\usepackage{color}
\usepackage{authblk}

\usetikzlibrary{decorations.pathreplacing, calligraphy}

\newtheorem{theorem}{Theorem}[section]
\newtheorem{lemma}[theorem]{Lemma}

\newtheorem{claim}{Claim}[theorem]
\Crefname{claim}{Claim}{Claims}
\newtheorem{corollary}[theorem]{Corollary}
\Crefname{subsection}{Subsection}{Subsections}

\Crefname{conjecture}{Conjecture}{Conjectures}
\theoremstyle{definition}
\newtheorem{definition}[theorem]{Definition}

\expandafter\let\expandafter\oldproof\csname\string\proof\endcsname
\let\oldendproof\endproof
\renewenvironment{proof}[1][\proofname]{%
	\oldproof[\normalfont\bfseries #1]%
}{\oldendproof}
\newenvironment{subproof}[1][\normalfont\it Subproof]{%
	\begin{proof}[#1]%
	}{%
	\end{proof}%
}

\newcommand{\CN}{\operatorname{CN}}
\newcommand{\Ov}{\mathcal{O}}
\newcommand{\Sh}{\operatorname{Sh}}
\newcommand{\Adj}{\operatorname{Adj}}
\newcommand{\Cov}{\operatorname{Cov}}
\newcommand{\cl}{\operatorname{core}}

\usepackage{libertinus}

\title{A quasi-polynomial $\chi$-bound for $P_6$-free graphs}

	\author[1]{Shenwei Huang}

    \author[2]{Yidong Zhou}

    \author[3]{Yuwen Zhou}
	
	\affil[1] {School of Mathematical Sciences and LPMC, Nankai University, Tianjin 300071, P.R.~China. Email: {\tt shenweihuang@nankai.edu.cn}.}
  
    \affil[2] {College of Computer Science, Nankai University, Tianjin 300350, China.}

    \affil[3] {Irving K. Barber Faculty of Science, University of British Columbia Okanagan, Kelowna, BC V1V 1V7, Canada.}

\begin{document}

\maketitle

\begin{abstract}
In this paper we prove that every graph $G$ with no induced six-vertex path satisfies $\chi(G)\le \omega(G)^{c\log_2\omega(G)}$ for $c=45/\log_2(4/3)<109$. 
The best previous bound was exponential in $\omega(G)$.
Our result is the first quasi-polynomial bound for $P_6$-free graphs and generalizes the bound for $P_5$-free graphs by Scott, Seymour, and Spirkl [Polynomial bounds for chromatic number. {IV}. {A} near-polynomial bound for excluding the five-vertex path.
{\em Combinatorica}, 43(5):845--852, 2023].

\medskip

\end{abstract}

\maketitle

\section{Introduction}

All graphs in this paper are finite and simple. A class of graphs is \emph{$\chi$-bounded} if there is a function $f$ such that $\chi(G)\le f(\omega(G))$ for every graph $G$ in the class, and \emph{polynomially $\chi$-bounded} if $f$ can be taken to be a polynomial. For a graph $H$, a graph is \emph{$H$-free} if it has no induced subgraph isomorphic to $H$, and $P_t$ denotes the path on $t$ vertices.

Gy\'arf\'as~\cite{gyarfas1975,gyarfas1987} proved that for every $t$ the class of $P_t$-free graphs is $\chi$-bounded, with $\chi(G)\le (t-1)^{\omega(G)-1}$. The base of the exponent was improved by Gravier, Ho\`ang and Maffray~\cite{ghm2003} and recently by Oum~\cite{oum2026}, but for every $t\ge 6$ all known bounds have remained exponential in $\omega(G)$. Whether the class of $P_t$-free graphs is polynomially $\chi$-bounded is a well-known open problem, see the survey of Scott and Seymour~\cite{ss2020}. 
The case $t=4$ is trivial ($P_4$-free graphs are perfect). 
For $t=5$, Scott, Seymour and Spirkl~\cite{sss2023} proved that every $P_5$-free graph $G$ with $\omega(G)\ge3$ satisfies $\chi(G)\le\omega(G)^{\log_2\omega(G)}$.
Nguyen~\cite{nguyen2025pairs} improved this bound to $\chi(G)\le\omega(G)^{O(\log\omega(G)/\log\log\omega(G))}$.
He subsequently proved that the class of $P_5$-free graphs is polynomially $\chi$-bounded~\cite{nguyen2025p5}.

Our main result is a quasi-polynomial bound for $t=6$.

\begin{theorem}\label{thm:main}
Let $c=45/\log_2(4/3)$. For every $P_6$-free graph $G$,
$
\chi(G)\le \omega(G)^{c\log_2\omega(G)}$.

\end{theorem}

Let $\chi(X)$ and $\omega(X)$ denote $\chi(G[X])$ and $\omega(G[X])$. 
The heart of the paper is the following polynomial statement.

\begin{theorem}\label{thm:stable}
Let $G$ be a $P_6$-free graph, $R\subseteq V(G)$ be a stable set, and $Y\subseteq V(G)\setminus R$ be nonempty with every vertex of $Y$ having a neighbor in $R$. Then
$\chi(Y)\le 55\,\omega(Y)^{10}\max_{r\in R}\chi(N(r)\cap Y)$.

\end{theorem}

 For $A\subseteq V(G)$, let $N(A)$ be the set of vertices outside $A$ with a neighbor in $A$. We have the following corollary the above theorem.

\begin{corollary}\label{cor:blob}
Let $G$ be a $P_6$-free graph, and let $A_1,\dots,A_k$ be pairwise disjoint, pairwise anticomplete subsets of $V(G)$ such that each $G[A_i]$ is connected. Then 
$$\chi\bigl(N(A_1\cup\dots\cup A_k)\bigr)\le 55\,\omega(G)^{10}\max_{1\le i\le k}\chi(N(A_i)).$$
\end{corollary}

The paper is organized as follows.
\Cref{sec:prelim} collects notation and the results we cite. \Cref{sec:tools} introduces dominated pairs and proves the basic tools used for \cref{thm:stable}. \Cref{sec:stablestars} proves the cubic bound for stable neighborhoods.
\cref{sec:overlap} colors the overlap graph and completes the proofs of \cref{thm:stable} and \cref{cor:blob}, and \cref{sec:main} proves \cref{thm:main}.

\section{Preliminaries}\label{sec:prelim}

Let $G$ be a graph. We write $u\sim v$ if $u,v$ are adjacent and $u\not\sim v$ otherwise.
Let $N(v)$ be the set of neighbors of $v$ and $N[v]=N(v)\cup\{v\}$. 
For $A\subseteq V(G)$, $N(A)$ is the set of vertices outside $A$ with a neighbor in $A$, and $N_C(A)=N(A)\cap C$. Disjoint sets $A,B$ are \emph{complete} (\emph{anticomplete}) if every vertex of $A$ is adjacent (non-adjacent) to every vertex of $B$; a vertex $v\notin A$ is complete or anticomplete to $A$ if $\{v\}$ is. A vertex $v\notin A$ is \emph{mixed} on $A$ if it has both a neighbor and a non-neighbor in $A$. We write
$\CN(A)=\{v\in V(G)\setminus A:\ v \text{ is complete to } A\}$.
A set $A$ is \emph{connected} if $G[A]$ is connected and \emph{co-connected} if the complement of $G[A]$ is connected. An induced path $v_1\hbox{-}v_2\hbox{-}\cdots\hbox{-}v_k$ has $v_i\sim v_{i+1}$
for each $1\le i\le k-1$ and no other edges among its vertices. All logarithms are to base $2$.

We present some known results we shall use later.

\begin{lemma}[Gy\'arf\'as~\cite{gyarfas1987}]\label{lem:gyarfas}
Every $P_t$-free graph $G$ satisfies $\chi(G)\le (t-1)^{\omega(G)-1}$.
\end{lemma}

For a graph and $k\ge1$, a \emph{$k$-covering blockade} is a sequence $(D_1,\dots,D_k,E)$ of nonempty disjoint sets of vertices such that every vertex of $D_i$ ($2\le i\le k$) has a neighbor in $D_{i-1}$ and no neighbor in $D_1\cup\dots\cup D_{i-2}$, and $E$ is anticomplete to $D_1\cup\dots\cup D_{k-1}$.
The following lemma is a simple corollary of the $k=5$ case of Nguyen's covering lemma~\cite[Lemma~3.3]{nguyen2025pairs}.

\begin{lemma}\label{lem:nguyen}
Let $G$ be a non-complete $P_6$-free graph with $\omega(G)=w\ge 2$. Then $G$ has an anticomplete pair $(A,B)$ such that $A$ and $B$ are connected and $\chi(A),\chi(B)\ge w^{-40}\chi(G)$.
\end{lemma}

\begin{proof}
By~\cite[Lemma~3.3]{nguyen2025pairs} with $k=5$, either $G$ has an anticomplete pair $(A,B)$ so that $\chi(A),\chi(B)\ge w^{-40}\chi(G)$ or it has a $5$-covering blockade $(D_1,\dots,D_5,E)$ with $\chi(D_5),\chi(E)\ge w^{-30}\chi(G)$. 
In the second case, suppose some $v\in E$ has a neighbor $d_5\in D_5$. Choose $d_4\in D_4$ adjacent to $d_5$, $d_3\in D_3$ adjacent to $d_4$, $d_2\in D_2$ adjacent to $d_3$, and $d_1\in D_1$ adjacent $d_2$. 
Then $v\hbox{-}d_5\hbox{-}d_4\hbox{-}d_3\hbox{-}d_2\hbox{-}d_1$ is an induced $P_6$.  
So $D_5$ is anticomplete to $E$, and $(D_5,E)$ is an anticomplete pair. 
Finally, replacing $A$ and $B$ by components of $G[A]$ and $G[B]$ of maximum chromatic number keeps the bounds and makes both sides connected.
\end{proof}

We use the modular decomposition of Gallai~\cite{gallai1967}. A set $M\subseteq X$ is a \emph{module} of $G[X]$ if every vertex of $X\setminus M$ is complete or anticomplete to $M$; it is \emph{strong} if for every module $M'$, either $M\cap M'=\emptyset$ or one of $M$ and $M'$ is contained in the other. The strong modules of $G[X]$ form a tree under inclusion, called the \emph{modular decomposition tree}. The children of a strong module $M$ with $|M|\ge2$ are the {\em maximal strong modules} of $G[M]$.

\begin{lemma}[Gallai~\cite{gallai1967}]\label{lem:gallai}
Let $X$ be a set with $|X|\ge 2$.
\begin{enumerate}
\item If $X$ is not connected, the maximal strong modules of $G[X]$ are the vertex sets of its components.
\item If $X$ is connected but not co-connected, the maximal strong modules of $G[X]$ are the vertex sets of the components of the complement of $G[X]$, and they are pairwise complete.
\item If $X$ is connected and co-connected, the maximal strong modules of $G[X]$ partition $X$. Moreover, any two of them are complete or anticomplete to each other, and the \emph{quotient graph} (one vertex per maximal strong module, adjacency as in $G$) is connected, and every module $M\ne X$ of $G[X]$ is contained in a maximal strong module.
\end{enumerate}
\end{lemma}

An ordering of the vertices of a graph is \emph{perfect} if the greedy coloring along it uses $\omega(H)$ colors on every induced subgraph $H$ (with the induced ordering).

\begin{lemma}[Chv\'atal~\cite{chvatal1984}]\label{lem:chvatal}
In a $P_4$-free graph every vertex ordering is perfect.
\end{lemma}

\begin{lemma}\label{lem:contract}
Let $t\ge1$ and $G$ be a $P_t$-free graph. 
Let $A_1,\dots,A_k$ be pairwise disjoint and pairwise anticomplete connected sets, and let $U\subseteq V(G)\setminus(A_1\cup\dots\cup A_k)$. Let $H$ be the graph obtained from $G[U]$ by adding pairwise non-adjacent new vertices $z_1,\dots,z_k$ with $N_H(z_i)=U\cap N(A_i)$. Then $H$ is $P_t$-free.
\end{lemma}

\begin{proof}
Let $P$ be an induced path of  $H$. If $z_i$ is an end of $P$ and $a$ is its unique neighbor on $P$, we replace $z_i$ with a neighbor of $a$ in $A_i$.
If $z_i$ has two neighbors $a,b$ on $P$, then we replace $z_i$ with the interior of a shortest path in $A_i\cup \{a,b\}$ connecting $u$ and $v$. 
It follows that the resulting path is still induced and has as many vertices as $P$.
\end{proof}

\begin{lemma}[Seinsche~\cite{seinsche1974}]\label{lem:P4}
    Every $P_4$-free graph with at least two vertices is disconnected or disconnected complement.
\end{lemma}

\section{Dominated pairs}\label{sec:tools}

\begin{definition}
A \emph{dominated pair} $(Y,R)$ in a graph $G$ consists of disjoint sets $Y,R\subseteq V(G)$ such that $G[Y\cup R]$ is $P_6$-free, $R$ is stable, and every vertex of $Y$ has a neighbor in $R$. The vertices of $R$ are called \emph{points}. For $y\in Y$ the \emph{type} of $y$ is $S(y)=N(y)\cap R$. For $r\in R$ the \emph{star} of $r$ is $I_r=N(r)\cap Y$. We write
$M(Y,R)=\max_{r\in R}\chi(I_r)$.

Two vertices $y,y'\in Y$ form a \emph{far pair} if $y\not\sim y'$ and $S(y)\cap S(y')=\emptyset$. For $z\in Y$ we put
\[
Q(z)=\{y\in Y\setminus N[z]:\ S(y)\cap S(z)\ne\emptyset\},\]
\[
F(z)=\{y\in Y\setminus N[z]:\ S(y)\cap S(z)=\emptyset\}.
\]
Clearly, $Y=N_Y[z]\cup Q(z)\cup F(z)$. For $X\subseteq Y$ we put $Q_X(z)=Q(z)\cap X$.
\end{definition}

If $(Y,R)$ is a dominated pair and $Y'\subseteq Y$, then $(Y',R)$ is a dominated pair with $M(Y',R)\le M(Y,R)$. 
In this section, we stipulate that $(Y,R)$ is a dominated pair, $M=M(Y,R)$,  $w=\omega(Y)$, and $G=G[Y\cup R]$.

\subsection{Non-neighbors that share a point}

For $z\in Y$ and $y\in Q(z)$, we call $S(y)\cap S(z)$ the \emph{trace} of $y$ and $S(y)\setminus S(z)$ its \emph{outer part}. Two traces \emph{cross} if neither contains the other.

\begin{lemma}\label{lem:crossing}
Let $z\in Y$ and $y,y'\in Q(z)$ be non-adjacent with crossing traces. Then $y$ and $y'$ have the same outer part.
\end{lemma}

\begin{proof}
Choose $r\in S(y)\cap S(z)\setminus S(y')$ and $r'\in S(y')\cap S(z)\setminus S(y)$. If some $s\in S(y)$ lies outside $S(z)\cup S(y')$, then $s\hbox{-}y\hbox{-}r\hbox{-}z\hbox{-}r'\hbox{-}y'$ is an induced $P_6$.
Hence the outer part of $y$ is contained in outer part of $y'$. By symmetry, the two outer parts are equal.
\end{proof}

\begin{lemma}\label{lem:Q}
Let $X\subseteq Y$ and $z\in X$ be such that no vertex of $X$ has type strictly contained in $S(z)$. Then every $D\subseteq Q_X(z)$ satisfies $\chi(D)\le 2M\omega(D)$.
\end{lemma}

\begin{proof}
We use induction on $\omega(D)$. 
Choose $a\in D$ whose trace is minimal under inclusion among  vertices of $D$.
Choose a point $x$ in the trace of $a$, and if the outer part of $a$ is nonempty, choose a point $s$ in it. 
We claim that $D\setminus N(a)\subseteq I_x\cup I_s$ (with $I_s=\emptyset$ if $s$ is undefined). Let $y\in D\setminus N(a)$. If the trace of $y$ contains that of $a$, then $y\in I_x$. 
By the choice of $a$, 
the trace of $y$ is not strictly inside that of $a$. 
So the two traces cross and \cref{lem:crossing} shows that $y$ and $a$ have the same outer part. If it is nonempty, $y\in I_s$. If it is empty, then $S(y)\subseteq S(z)$ and so $S(y)=S(z)$ by the hypothesis on $X$. Then the trace of $y$ is $S(z)$, which contains the trace of $a$. This is a contradiction. 
This proves the claim, and so $\chi(D\setminus N(a))\le 2M$. Since $D\cap N(a)\subseteq Q_X(z)$ has clique number at most $\omega(D)-1$, induction gives $\chi(D)\le 2M+2M(\omega(D)-1)=2M\omega(D)$.
\end{proof}

\begin{theorem}\label{thm:nofar}
If $(Y,R)$ has no far pair, then $\chi(Y)\le (w^2+w)M+w$.
\end{theorem}

\begin{proof}
Induction on $w$. Choose $z\in Y$ whose type is minimal under inclusion. As there is no far pair, $Y\setminus N[z]=Q_Y(z)$. \cref{lem:Q} with $X=Y$ gives $\chi(Q_Y(z))\le 2wM$. The set $N(z)\cap Y$ has no far pair and clique number at most $w-1$. By inductive hypothesis, $\chi(Y)\le \bigl(((w-1)^2+(w-1))M+(w-1)\bigr)+1+2wM=(w^2+w)M+w$.
\end{proof}

\subsection{Capped pairs}

A \emph{cap} of $(Y,R)$ is a vertex $c\notin Y\cup R$ that is adjacent to every vertex of $R$ and to no vertex of $Y$ such that $G[Y\cup R\cup\{c\}]$ is $P_6$-free.

\begin{lemma}\label{lem:capkey}
Let $c$ be a cap of $(Y,R)$, $a,b,b'\in Y$ with $b\sim b'$, $a\not\sim b$, $a\not\sim b'$, and $S(b)\cup S(b')$ disjoint from $S(a)$. Then $S(b)=S(b')$.
\end{lemma}

\begin{proof}
If $\ell_b\in S(b)\setminus S(b')$ and $\ell_a\in S(a)$, then $b'\hbox{-}b\hbox{-}\ell_b\hbox{-}c\hbox{-}\ell_a\hbox{-}a$ is an induced $P_6$. So $S(b)\subseteq S(b')$, and the reverse inclusion holds by symmetry.
\end{proof}

\begin{theorem}\label{thm:cap}
If $(Y,R)$ has a cap, then $\chi(Y)\le (w^2+2w)M$.
\end{theorem}

\begin{proof}
Induction on $w$. Choose $a\in Y$ whose type is minimal under inclusion. Then $Y\setminus N[a]=Q(a)\cup F(a)$. By \cref{lem:Q} with $X=Y$, $\chi(Q_Y(a))\le 2wM$. By \cref{lem:capkey}, adjacent vertices of $F(a)$ have equal types, and so all vertices in each component of $G[F(a)]$ have the same type. This implies that for each component of $G[F(a)]$, there exists a vertex in $R$ complete to it. So
$\chi(F(a)\cup\{a\})\le M$. 
The pair $(N(a)\cap Y,R)$ has the same cap and clique number at most $w-1$. So $\chi(Y)\le ((w-1)^2+2(w-1))M+2wM+M=(w^2+2w)M$.
\end{proof}

\begin{corollary}\label{cor:capt}
For every $t\in Y$, $\chi(Q(t))\le (w^2+4w)M$.
\end{corollary}

\begin{proof}
Let $D(t)$ be the set of non-neighbors of $t$ in $Y$ whose type is contained in $S(t)$, and $P(t)\subseteq D(t)$ be the set of those whose type is strictly contained in $S(t)$. The pair $(D(t),S(t))$ is a dominated pair whose types are the original ones and whose stars lie inside the original stars, and $t$ is a cap of it. So $\chi(D(t))\le (w^2+2w)M$ by \cref{thm:cap}. Let $X=\{t\}\cup (Y\setminus (N[t]\cup P(t)))$. No vertex of $X$ has type strictly inside $S(t)$ and $Q_X(t)=Q(t)\setminus P(t)$.
\cref{lem:Q} gives $\chi(Q(t)\setminus P(t))\le 2wM$. As $P(t)
\subseteq D(t)$, $\chi(Q(t))\le 2wM+(w^2+2w)M$.
\end{proof}

\section{Stable stars}\label{sec:stablestars}

In this section we prove the following.

\begin{theorem}\label{thm:stablestars}
Let $(Y,R)$ be a dominated pair in which every star is stable. Then $\chi(Y)\le 5\,\omega(Y)^3$.
\end{theorem}

Throughout this section all dominated pairs have stable stars. We use repeatedly that
\begin{equation}\label{eq:stab}
\text{if } y\sim y' \text{ then } S(y)\cap S(y')=\emptyset.
\end{equation}

We need the following coloring lemma; its proof is in \cref{app:cg}.

\begin{theorem}\label{thm:cg}
Let $(Y,R)$ be a dominated pair with stable stars in which every vertex has exactly one point. Then $\chi(Y)\le\omega(Y)^2$.
\end{theorem}

\subsection{Proof of \texorpdfstring{\cref{thm:cg}}{Theorem 4.2}}\label{app:cg}

Let $U$ be a graph with a proper coloring $\gamma$ with color classes $V_j$. Let $U^+$ be obtained from $U$ by adding, for each color $j$, a new vertex $z_j$ with $N(z_j)=V_j$ such that The new vertices are pairwise non-adjacent. Let $D=D(U,\gamma)$ be the graph on $V(U)$ in which $u,v$ are adjacent if $u\not\sim v$ and $\gamma(u)\ne\gamma(v)$. 
We speak of \emph{$D$-edges}, \emph{$D$-neighbors} and so on.

\Cref{thm:cg} is the case $U=G[Y]$, with $\gamma(y)$ the unique point of $y$: the classes are the stars, $U^+=G[Y\cup R]$ (after deleting points with empty star) is $P_6$-free, and the stars are stable, so $\gamma$ is proper. It therefore follows from the next theorem.

\begin{theorem}\label{thm:cgfull}
Let $\gamma$ be a proper coloring of $U$ such that $U^+$ is $P_6$-free. Then $\chi(U)\le\omega(U)^2$.
\end{theorem}

The condition is hereditary: for $U'\subseteq V(U)$, $(U[U'])^+$ is an induced subgraph of $U^+$ (keep the $z_j$ with $V_j\cap U'\ne\emptyset$). The main step is the following.

\begin{theorem}\label{thm:cgsplit}
Let $\gamma$ be a proper coloring of $U$ such that $U^+$ is $P_6$-free, with $|V(U)|\ge2$. Then $U$ or $D(U,\gamma)$ is disconnected.
\end{theorem}

We use the argument of Gy\'arf\'as~\cite[Proposition~5.8]{gyarfas1987} to prove Theorem \ref{thm:cgfull} assuming Theorem \ref{thm:cgsplit}.

\begin{proof}[Proof of \cref{thm:cgfull} from \cref{thm:cgsplit}]
We claim that there is a graph $H$ on $V(U)$, built from single vertices by disjoint unions and joins such that 
for every two vertices $u,v\in V(H)$, 
\[v\sim u \text{ in } U \Leftrightarrow \gamma(v)\neq \gamma(u) \text{ and } v\sim u \text{ in } H.\]
This is clear if $|V(U)|=1$. 
By Theorem \ref{thm:cgsplit}, $U$ or $D$ is disconnected.
If $U$ is disconnected, take the disjoint union of the graphs for the components. If $D$ is disconnected with components $W_1,\dots,W_s$, then any two vertices in different $W_i$ of different colors are adjacent in $U$; take the join of the graphs for the $W_i$. The claim follows by induction.

So $H$ is $P_4$-free and so it is the comparability graph of the strict partial order $<_1$. Let $u<_2v$ if $\gamma(u)<\gamma(v)$. Every edge of $U$ is comparable in both orders. Let $E^+$ be the set of edges of $U$ on which the two orders agree and $E^-$ the set on which they disagree. Then $E^+$ and $E^-$ are the comparability graphs and so are perfect.  So both have proper colorings with $\omega(U)$ colors. So the product coloring is proper on $U$.
\end{proof}

From now on, $\gamma$ is a proper coloring of $U$ with $U^+$ being $P_6$-free. Fix a color $j$. For $u\notin V_j$ let $\nu(u)=V_j\setminus N(u)$, and for $y\in V_j$ let $\nu(y)=\{y\}$. Let
\[
C_j=\{u\notin V_j:\nu(u)=\emptyset\},\qquad R_j=\{u\notin V_j:\nu(u)=V_j\},\qquad J_j=V(U)\setminus(C_j\cup R_j).
\]
A \emph{rainbow} path is one whose vertices have pairwise different colors.

\begin{lemma}\label{lem:cgfacts}
The following holds.
\begin{enumerate}[label=\upshape(\alph*)]
\item There is no rainbow induced $P_4$.
\item Let $y\in V_j$, $p\sim y$ and $q\notin V_j\cup N(y)$ with $q\sim p$ or $\gamma(q)=\gamma(p)$. Then $V_j\subseteq N(p)\cup N(q)$.
\item Let $u,u'\notin V_j\cup C_j$ with $\nu(u)\cap\nu(u')=\emptyset$. Then $u\sim u'$ if and only if $\gamma(u)\ne\gamma(u')$.
\item The sets $\nu(u)$, $u\in V(U)$, form a laminar family.
\item If $\emptyset\ne\nu(u)\subsetneq\nu(u')$, then $u\not\sim u'$ and $\gamma(u)\ne\gamma(u')$.
\end{enumerate}
\end{lemma}

\begin{proof}
(a) If $a\hbox{-}b\hbox{-}c\hbox{-}d$ is such a path, then $z_{\gamma(a)}\hbox{-}a\hbox{-}b\hbox{-}c\hbox{-}d\hbox{-}z_{\gamma(d)}$ is an induced $P_6$.

(b) Let $x\in V_j$ be non-adjacent to $p$ and $q$. Then $x\ne y$. If $q\sim p$, then $x\hbox{-}z_j\hbox{-}y\hbox{-}p\hbox{-}q\hbox{-}z_{\gamma(q)}$ is an induced $P_6$. If $\gamma(q)=\gamma(p)$, then $x\hbox{-}z_j\hbox{-}y\hbox{-}p\hbox{-}z_{\gamma(p)}\hbox{-}q$ is. (Here $\gamma(p)\ne j$ as $p\sim y$ and $\gamma(q)\ne j$.)

(c) If $\gamma(u)\ne\gamma(u')$ and $u\not\sim u'$, choose $x\in\nu(u')$ and $y\in\nu(u)$. Then $u\sim x$ and $u'\sim y$ and so $z_{\gamma(u)}\hbox{-}u\hbox{-}x\hbox{-}z_j\hbox{-}y\hbox{-}u'$ is an induced $P_6$. The converse holds as $\gamma$ is proper.

(d) Suppose $a\in\nu(u)\setminus\nu(u')$, $b\in\nu(u')\setminus\nu(u)$ and $x\in\nu(u)\cap\nu(u')$. Then $u,u'\notin V_j$. If $u\sim u'$ or $\gamma(u)=\gamma(u')$, then (b) with $p=u$, $y=b$, $q=u'$ gives $V_j\subseteq N(u)\cup N(u')$, contrary to the existence of $x$. Otherwise $z_{\gamma(u)}\hbox{-}u\hbox{-}b\hbox{-}z_j\hbox{-}a\hbox{-}u'$ is an induced $P_6$.

(e) Since $|\nu(u')|\ge 2$, $u'\notin V_j$.
If $u\in V_j$, this is clear. So $u\notin V_j$.
Choose $a\in\nu(u')\setminus\nu(u)$. If $u\sim u'$ or $\gamma(u)=\gamma(u')$, then (b) with $p=u$, $y=a$, $q=u'$ gives $\nu(u)\cap\nu(u')=\emptyset$. So $\nu(u)=\emptyset$, a contradiction.
\end{proof}

\begin{lemma}\label{lem:top}
Suppose $C_j=\emptyset$ and $|V(U)|\ge2$. Then $U$ or $D$ is disconnected. Moreover, if $|V_j|\ge 2$, then there exists a partition of $V(U)=R_j\cup B_1\cup \cdots \cup B_r$ with $r\ge 2$ such that 
\begin{itemize}
    \item $R_j$ is anticomplete to every $B_s$, and 
    \item any two vertices in different sets $B_s,B_{s'}$ are adjacent if and only if they have different colors, and
    \item every vertex $u$ of $B_s$ has $\nu(u)\subseteq M_s$ for a partition $V_j=M_1\cup\dots\cup M_r$.
\end{itemize}
\end{lemma}

\begin{proof}
If $V_j=\{y\}$, every other vertex has $\nu=\{y\}$ and so $y$ is isolated in $U$. Let $|V_j|\ge2$. By \cref{lem:cgfacts} (d), the family $\{\nu(u)\}$ is laminar. Its members other than $V_j$ include all singletons of $V_j$ and so its maximal members $M_1,\dots,M_r$ other than $V_j$ partition $V_j$, where $r\ge2$. 
Let $B_s=\{u:\nu(u)\subseteq M_s\}$. As $C_j=\emptyset$, every vertex lies in $R_j$ or in some $B_s$. By \cref{lem:cgfacts} (e), $R_j$ is anticomplete to every $B_s$. For $u\in B_s$, $u'\in B_{s'}$ with $s\ne s'$, \cref{lem:cgfacts} (c) gives the stated adjacency if $u,u'\notin V_j$, and it is clear otherwise. If $R_j\ne\emptyset$, then $U$ is disconnected. Otherwise there is no $D$-edge between $B_1$ and $B_2\cup\dots\cup B_r$.
\end{proof}

Applied to $U-C_j$ (in which no vertex is complete to $V_j$), \cref{lem:top} shows: $R_j$ is anticomplete to $J_j$, and by \cref{lem:cgfacts}(e) no color occurs both in $R_j$ and in $J_j\setminus V_j$.

\begin{lemma}\label{lem:cgmore}
Let $w,w_1,w_2\in C_j$ and $r\in R_j$.
\begin{enumerate}[label=\upshape(\alph*)]
\item If $w_1w_2$ and $xw_1$ are $D$-edges with $x\notin C_j\cup V_j$, then $xw_2$ is a $D$-edge.
\item If $w\sim r$, then $w$ has no $D$-neighbor in $J_j\setminus V_j$.
\item If $\gamma(w)=\gamma(r)$, then $w$ has no $D$-neighbor in $J_j\setminus V_j$.
\end{enumerate}
\end{lemma}

\begin{proof}
(a) Choose $y\in\nu(x)$. If $x\sim w_2$, then $w_1\hbox{-}y\hbox{-}w_2\hbox{-}x$ is a rainbow induced $P_4$, contradicting Theorem \ref{lem:cgfacts} (a). If $\gamma(x)=\gamma(w_2)=m$, then $z_{\gamma(w_1)}\hbox{-}w_1\hbox{-}y\hbox{-}w_2\hbox{-}z_m\hbox{-}x$ is an induced $P_6$.

(b) Let $x$ be such a $D$-neighbor and $y\in N(x)\cap V_j$. Then $r\hbox{-}w\hbox{-}y\hbox{-}x$ is a rainbow induced $P_4$ ($\gamma(r)\ne\gamma(x)$ since $R_j$ and $J_j\setminus V_j$ have no common color).

(c) Let $x$ be such a $D$-neighbor and $y\in N(x)\cap V_j$. Then $z_{\gamma(x)}\hbox{-}x\hbox{-}y\hbox{-}w\hbox{-}z_{\gamma(w)}\hbox{-}r$ is an induced $P_6$.
\end{proof}

\begin{lemma}\label{lem:rz}
Suppose $R_j=\emptyset$ and $C_j\ne\emptyset$, and $|V(U)|\ge2$. Then $D$ is disconnected.
\end{lemma}

\begin{proof}
If $V_j=\{y\}$, then every vertex outside $V_j$ is adjacent to $y$ and so $y$ is isolated in $D$. Let $|V_j|\ge2$ and apply \cref{lem:top} to $U-C_j$: $V(U)\setminus C_j=B_1\cup\dots\cup B_r$ with $r\ge2$. Let $K$ be the vertex set of a component of $D[C_j]$. Vertices of $C_j$ have no $D$-neighbor in $V_j$. By \cref{lem:cgmore}~(a), all vertices of $K$ have the same set $N_K$ of $D$-neighbors outside $C_j$. We claim that $N_K$ meets at most one $B_s$. Suppose $x_1\in N_K\cap B_1$ and $x_2\in N_K\cap B_2$ are $D$-adjacent to $w\in K$. Choose $y_1\in\nu(x_1)\subseteq M_1$ and $y_2\in\nu(x_2)\subseteq M_2$. Then $x_1\sim y_2$ and $x_2\sim y_1$. If $\gamma(x_1)\ne\gamma(x_2)$, then $x_1\sim x_2$ by Lemma \ref{lem:top} and $x_2\hbox{-}x_1\hbox{-}y_2\hbox{-}w$ is a rainbow induced $P_4$. If $\gamma(x_1)=\gamma(x_2)=\ell$, then $z_{\gamma(w)}\hbox{-}w\hbox{-}y_2\hbox{-}x_1\hbox{-}z_\ell\hbox{-}x_2$ is an induced $P_6$. This proves the claim. Now put each $K$ into the part $B_s$ that $N_K$ meets (or into a part of its own if $N_K=\emptyset$). There is no $D$-edge between different parts, and there are at least two parts.
\end{proof}

\begin{proof}[Proof of \cref{thm:cgsplit}]
Suppose $(U,\gamma)$ is a counterexample with $|V(U)|$ minimum; so $U$ and $D$ are connected, and by minimality every induced subgraph on at least two vertices has $U$ or $D$ disconnected. If every color class is a single vertex, then by \cref{lem:cgfacts} (a) $U$ is $P_4$-free, and $D$ is its complement. This contradicts Lemma \ref{lem:P4}. So choose $j$ with $|V_j|\ge2$. By \cref{lem:top}, $C_j\ne\emptyset$. By \cref{lem:rz}, $R_j\ne\emptyset$.

Suppose $J_j\ne V_j$. Let $U'=C_j\cup R_j\cup V_j$, a proper subset of $V(U)$. The graph $U[U']$ is connected: $C_j$ is complete to $V_j$, and every component of $U[R_j]$ has a neighbor outside $R_j$ (as $U$ is connected), which lies in $C_j$ because $R_j$ is anticomplete to $J_j$. So $D(U[U'])$ is disconnected. 
Every vertex of $R_j$ is $D$-adjacent to every vertex of $V_j$ and so $R_j\cup V_j$ lies in one component of $D(U[U'])$, and some other component $K$ lies inside $C_j$.
Also,
$K$ has no $D$-neighbor in $R_j\cup V_j$ and is a component of $D[C_j]$. If some $w\in K$ is adjacent to a vertex of $R_j$, then by \cref{lem:cgmore} (b) $w$ has no $D$-neighbor in $J_j\setminus V_j$, and by \cref{lem:cgmore} (a) and connectivity of $K$, any vertex in $K$ has no $D$-neighbor in $J_j\setminus V_j$. So $K$ is a component of $D$, a contradiction. 
Thus, we may assume that $K$ is anticomplete to $R_j$. 
So every $w\in K$ has the color of every vertex of $R_j$. 
Since $D$-edges join different colors and $K$ is a component of $D[C_j]$, $K=\{w\}$. Then by \cref{lem:cgmore} (c), $w$ is isolated in $D$, a contradiction.

So $J_j=V_j$ and every vertex outside $V_j$ is complete or anticomplete to $V_j$. Then two vertices $y,y'\in V_j$ have the same neighbors in $U$ and in $D$, and $U-y'$ is a smaller counterexample, a contradiction.
\end{proof}

\subsection{Proof of \texorpdfstring{\cref{thm:stablestars}}{Theorem 4.1}}

We say that $R$ is \emph{minimal} if no proper subset $R'\subsetneq R$ makes $(Y,R')$ a dominated pair. 
If $R$ is minimal, then every point $r$ has a \emph{private} vertex $p\in Y$, that is, one with $S(p)=\{r\}$. Let $Y_{\ge2}$ be the set of vertices with at least two points. For every $u\in Y_{\ge2}$ we fix two points $r_u^1,r_u^2\in S(u)$ and private vertices $p_u^1,p_u^2$ of them, and call
\[
X_u=\{p_u^1,r_u^1,u,r_u^2,p_u^2\},\qquad \cl(u)=\{u,p_u^1,p_u^2\}
\]
the \emph{frame} and the \emph{core} of $u$. By~\eqref{eq:stab}, $u\not\sim p_u^i$, and $p_u^1\not\sim r_u^2$. So $G[X_u]$ is the path $p_u^1\hbox{-}r_u^1\hbox{-}u\hbox{-}r_u^2\hbox{-}p_u^2$, possibly with the extra edge $p_u^1p_u^2$. In particular, $X_u$ is connected. Let $\Phi$ be the \emph{far graph}: its edges are the far pairs inside $Y_{\ge2}$ and its vertices are the vertices of $Y_{\ge2}$ that lie in at least one far pair inside $Y_{\ge 2}$.

\begin{lemma}\label{lem:frames}
If $ab$ is an edge of $\Phi$, then $X_a$ and $X_b$ are disjoint and anticomplete.
\end{lemma}

\begin{proof}
The sets are disjoint because $S(a)\cap S(b)=\emptyset$ and privates of different points are different. Points of $X_a$ are not adjacent to $b$ (disjoint types) or to $p_b^j$ (privacy), and symmetrically. It remains to consider pairs inside $Y$. If $a\sim p_b^1$, then $r_a^1\hbox{-}a\hbox{-}p_b^1\hbox{-}r_b^1\hbox{-}b\hbox{-}r_b^2$ is an induced $P_6$. If $p_a^i\sim p_b^j$, then $a\hbox{-}r_a^i\hbox{-}p_a^i\hbox{-}p_b^j\hbox{-}r_b^j\hbox{-}b$ is an induced $P_6$ (using the previous case for $a\not\sim p_b^j$ and $b\not\sim p_a^i$). The remaining cases are symmetric.
\end{proof}

A vertex $t$ \emph{touches} $X_u$ if $t\notin X_u$ and $t$ has a neighbor in $X_u$. If $ab\in E(\Phi)$ and $t$ touches both $X_a$ and $X_b$, we say that $t$ \emph{connects} $ab$.

\begin{lemma}\label{lem:kinds}
Let $ab\in E(\Phi)$ and $t$ connect $ab$. Then $t\in Y$, and for each $u\in\{a,b\}$ either $N(t)\cap X_u=\cl(u)$, or $N(t)\cap X_u=\{r_u^1,r_u^2\}$ and $p_u^1\not\sim p_u^2$.
\end{lemma}

\begin{proof}
Since a point adjacent to vertices of both frames would lie in $S(a)\cap S(b)=\emptyset$, $t\in Y$. By~\eqref{eq:stab} $t$ is complete to neither frame. So $t$ is mixed on the connected set $X_b$. Since $X_b$ is connected, there is an edge $yy'$ of $G[X_b]$ with $t\sim y$, $t\not\sim y'$. Let $x\in X_a\cap N(t)$ and let $x\hbox{-}x'\hbox{-}x''$ be an induced path in $G[X_a]$. If $t\not\sim x'$ and $t\not\sim x''$, then $x''\hbox{-}x'\hbox{-}x\hbox{-}t\hbox{-}y\hbox{-}y'$ is an induced $P_6$ by \cref{lem:frames}. So
\begin{equation}\label{eq:star}
\text{for every } x\in N(t)\cap X_a \text{ and every induced path } x\hbox{-}x'\hbox{-}x'' \text{ in } G[X_a],\ \ t \sim x' \text{ or } t\sim x''.
\end{equation}
By~\eqref{eq:stab}, if $t\sim r_a^i$ then $t\not\sim a$ and $t\not\sim p_a^i$. If $t\sim r_a^1$ and $t\not\sim r_a^2$, then the path $r_a^1\hbox{-}a\hbox{-}r_a^2$ violates~\eqref{eq:star}. If $t\sim r_a^1,r_a^2$, then $N(t)\cap X_a=\{r_a^1,r_a^2\}$. If $p_a^1\sim p_a^2$, then the path $r_a^1\hbox{-}p_a^1\hbox{-}p_a^2$ violates~\eqref{eq:star}. If $t$ is adjacent to neither point, then $N(t)\cap X_a\subseteq\cl(a)$ is nonempty. 
By applying~\eqref{eq:star} to
$p_a^i\hbox{-}r_a^i\hbox{-}a$ and $a\hbox{-}r_a^i\hbox{-}p_a^i$, $N(t)\cap X_a=\cl(a)$. The same holds for $b$.
\end{proof}

For $t\in Y$ let
\begin{gather*}
\Adj(t)=\{u\in Y_{\ge2}: t \text{ is complete to } \cl(u)\},\\
\Cov(t)=\{u\in Y_{\ge2}: t \text{ is complete to }X_u\setminus \cl(u)\},\\
\Sh(t)=\Adj(t)\cup\Cov(t).
\end{gather*}
We call $\Sh(t)$ the \emph{shadow} of $t$. The sets $\Adj(t)$ and $\Cov(t)$ are disjoint by~\eqref{eq:stab}.

\begin{lemma}\label{lem:shadow}
Let $t\in Y$.
\begin{enumerate}
\item If $v\in\Sh(t)$ and $vw\in E(\Phi)$, then $t\notin X_w$.
\item Let $vw\in E(\Phi)$. Then $t$ connects $vw$ if and only if $v,w\in\Sh(t)$.
\item If $G[Y]$ is connected, every edge of $\Phi$ is connected by some vertex.
\end{enumerate}
\end{lemma}

\begin{proof}
(a) As $t\in Y$, we must show $t\notin\cl(w)$. If $t=w$, then $v\in\Adj(w)$ contradicts $v\not\sim w$ and $v\in\Cov(w)$ contradicts $S(v)\cap S(w)=\emptyset$. If $t=p_w^i$, then $v\in\Adj(t)$ contradicts \cref{lem:frames}, and $v\in\Cov(t)$ is impossible as $|S(p_w^i)|=1$.

(b) If $t$ connects $vw$, \cref{lem:kinds} gives $v,w\in\Sh(t)$. Conversely, $t\notin X_v\cup X_w$ by~(a). By the definition of shadow, $t$ connects $vw$.

(c) Let $ab\in E(\Phi)$. As $G[Y]$ is connected, there is a path from $X_a$ to $X_b$. Let $x\hbox{-}v_1\hbox{-}\cdots\hbox{-}v_k\hbox{-}y$ be a shortest one, with $x\in X_a$, $y\in X_b$.  By \cref{lem:frames}, $k\ge1$. Suppose that $k\ge2$.
By~\eqref{eq:stab}, $v_1$ is mixed on $X_a$, namely there is an edge $xx'$ of $G[X_a]$ with $v_1\sim x$ and $v_1\not\sim x'$ (we may re-choose $x$). Similarly, there is an edge $yy'$ of $G[X_b]$ with $v_k\sim y$, $v_k\not\sim y'$. By the minimality of the path, $x'\hbox{-}x\hbox{-}v_1\hbox{-}\cdots\hbox{-}v_k\hbox{-}y\hbox{-}y'$ is an induced path with at least six vertices, a contradiction. So $k=1$ and $v_1$ connects $ab$.
\end{proof}

\begin{lemma}\label{lem:propagate}
Let $t$ connect $xy\in E(\Phi)$, and let $u\ne x$ be a $\Phi$-neighbor of $y$ with $ux\notin E(\Phi)$. Then $u\in\Sh(t)$.
\end{lemma}

\begin{proof}
By \cref{lem:shadow} (a), $t\notin X_u\cup X_x\cup X_y$. It suffices to show that $t$ touches $X_u$, for then $t$ connects $uy$ and \cref{lem:shadow} (b) applies. 
Suppose for a contradiction that $t$ has no neighbor in $X_u$. 
Since $u,x\in Y_{\ge2}$ do not form a far pair, either $u\sim x$, or $u\not\sim x$ and there is a point $q\in S(u)\cap S(x)$. By \cref{lem:kinds}, $t$ \emph{sees} $y$ ($N(t)\cap X_y=\cl(y)$) or \emph{covers} $y$ ($N(t)\cap X_y=\{r_y^1,r_y^2\}$). Put $(y',y'')=(y,r_y^1)$ in the first case and $(y',y'')=(r_y^1,y)$ in the second. In both cases $t\hbox{-}y'\hbox{-}y''$ is an induced path. The following are induced $P_6$'s:
\begin{itemize}
\item $u\sim x$, $t$ sees $x$: \ $r_u^1\hbox{-}u\hbox{-}x\hbox{-}t\hbox{-}y'\hbox{-}y''$;
\item $u\sim x$, $t$ covers $x$: \ $r_u^1\hbox{-}u\hbox{-}x\hbox{-}r_x^1\hbox{-}t\hbox{-}y'$;
\item $u\not\sim x$, $t$ sees $x$: \ $u\hbox{-}q\hbox{-}x\hbox{-}t\hbox{-}y'\hbox{-}y''$;
\item $u\not\sim x$, $t$ covers $x$, $q\notin S(t)$, and $r_x^i\notin S(u)$ for some $i$: \ $u\hbox{-}q\hbox{-}x\hbox{-}r_x^i\hbox{-}t\hbox{-}y'$.
\end{itemize}
In the remaining case $t$ covers $x$ and $t,u$ have a common point $q'$ ($q'=q$ or $q'=r_x^i$). As $t$ has no neighbor in $X_u$, $q'\notin\{r_u^1,r_u^2\}$ and $p_u^1\hbox{-}r_u^1\hbox{-}u\hbox{-}q'\hbox{-}t\hbox{-}y'$ is an induced $P_6$. In each case the non-edges follow from the stability of $R$, from~\eqref{eq:stab}, from the disjointness of the types of $\Phi$-adjacent vertices, from \cref{lem:frames} applied to the edges $uy$ and $xy$ of $\Phi$, from \cref{lem:kinds}, and from the assumption that $t$ has no neighbor in $X_u$. This contradiction proves the lemma.
\end{proof}

\begin{lemma}\label{lem:triangle}
Let $x,y,u$ be pairwise adjacent in $\Phi$. There do not exist distinct $t,t'\in Y\setminus(X_x\cup X_y\cup X_u)$ such that $t$ connects $xy$ and has no neighbor in $X_u$, and $t'$ connects $yu$ and has no neighbor in $X_x$.
\end{lemma}

\begin{proof}
The three frames are pairwise disjoint and anticomplete by \cref{lem:frames}. Since $t$ is mixed on the connected set $X_x$, there is an edge $\alpha\alpha'$ of $G[X_x]$ with $t\sim\alpha$, $t\not\sim\alpha'$. Similarly, there is an edge $\beta\beta'$ of $G[X_u]$ with $t'\sim\beta$, $t'\not\sim\beta'$. If $t\sim t'$, then $\alpha'\hbox{-}\alpha\hbox{-}t\hbox{-}t'\hbox{-}\beta\hbox{-}\beta'$ is an induced $P_6$. So $t\not\sim t'$. If $t$ and $t'$ have a common neighbor $m\in X_y$, then $\alpha'\hbox{-}\alpha\hbox{-}t\hbox{-}m\hbox{-}t'\hbox{-}\beta$ is an induced $P_6$. Otherwise, by \cref{lem:kinds}, one of $t,t'$ is complete to $\cl(y)$ and the other is adjacent to $r_y^1$ but not to $y$. Then $\alpha\hbox{-}t\hbox{-}y\hbox{-}r_y^1\hbox{-}t'\hbox{-}\beta$ or $\alpha\hbox{-}t\hbox{-}r_y^1\hbox{-}y\hbox{-}t'\hbox{-}\beta$ is an induced $P_6$.
\end{proof}

\begin{lemma}\label{lem:uct}
Suppose that $G[Y]$ is connected, and let $K$ be a connected subgraph of $\Phi$ with at least one edge. Then there is $t\in Y\setminus V(K)$ with $V(K)\subseteq\Sh(t)$.
\end{lemma}

\begin{proof}
We first prove the following claim.

\begin{claim}\label{clm:phi adjacent in K}
    Let $t\in Y$, $C\subseteq\Sh(t)$ be such that $\Phi[C]$ is connected, $u\in V(\Phi)\setminus\Sh(t)$ have a $\Phi$-neighbor $y\in C$, and $t'$ connect $yu$. Then every vertex in $C$ is $\Phi$-adjacent to $u$ and lies in $\Sh(t')$.
\end{claim}
\begin{subproof}[Proof of Claim \ref{clm:phi adjacent in K}]
    We go along a path $y=m_0,m_1,\dots,m_\ell$ of $\Phi[C]$ and prove, for $m=m_j$ and $m'=m_{j-1}$, first (i) $mu\in E(\Phi)$ and then (ii) $m\in\Sh(t')$. For $j=0$, (i) holds by hypothesis and (ii) by \cref{lem:shadow}~(b). For $j\ge1$: by \cref{lem:shadow} (b) $t$ connects $mm'$, and $u\ne m$ as $u\notin\Sh(t)$. If $mu\notin E(\Phi)$, then \cref{lem:propagate} (with $x=m$, $y=m'$) gives $u\in\Sh(t)$, a contradiction. This proves (i). For (ii), suppose $m\notin\Sh(t')$ and apply \cref{lem:triangle} to $(x,y,u)=(m,m',u)$. These are pairwise $\Phi$-adjacent. By \cref{lem:shadow} (a), $t\notin X_u$ and $t'\notin X_m$. The vertex $t$ connects $mm'$, and it has no neighbor in $X_u$, for otherwise it would connect $mu$ and  $u\in\Sh(t)$ by \cref{lem:shadow} (b). The vertex $t'$ connects $m'u$ (as $m',u\in\Sh(t')$), and it has no neighbor in $X_m$, for otherwise it would connect $mu$ and $m\in\Sh(t')$. Finally $t\ne t'$ because $u\in\Sh(t')\setminus\Sh(t)$. \Cref{lem:triangle} gives a contradiction, which proves (ii) and the claim.
\end{subproof}

Now order $V(K)$ as $v_1,\dots,v_n$ so that each $v_j$ with $j\ge2$ has a $K$-neighbor in $v_1,\dots,v_{j-1}$. By \cref{lem:shadow} (c) and (b), some $t_2$ has $v_1,v_2\in\Sh(t_2)$. 
Given $t_j$ with $\{v_1,\dots,v_j\}\subseteq\Sh(t_j)$, let $t_{j+1}=t_j$ if $v_{j+1}\in\Sh(t_j)$; otherwise let $t_{j+1}$ be a vertex connecting $yv_{j+1}$ for an earlier $K$-neighbor $y$ of $v_{j+1}$. The claim (with $C=\{v_1,\dots,v_j\}$) gives $\{v_1,\dots,v_{j+1}\}\subseteq\Sh(t_{j+1})$. Let $t=t_n$. If $t\in V(K)$, then $t$ has a $K$-neighbor $v\in\Sh(t)$, and \cref{lem:shadow} (a) gives $t\notin X_t$, which is a contradiction.
\end{proof}

\begin{lemma}\label{lem:2k2}
Let $ab,cd\in E(\Phi)$ with $a,b,c,d$ distinct, and suppose none of $ac,ad,bc,bd$ is an edge of $\Phi$. Then $a,b$ are both adjacent to both $c,d$.
\end{lemma}

\begin{proof}
By symmetry it suffices to show $a\sim c$. Suppose not,  let $q\in S(a)\cap S(c)$.

Suppose first $a\not\sim d$, and let $q''\in S(a)\cap S(d)$. Then $q''\ne q$. If $s\in S(c)\setminus S(a)$, then $s\hbox{-}c\hbox{-}q\hbox{-}a\hbox{-}q''\hbox{-}d$ is an induced $P_6$. If $s\in S(d)\setminus S(a)$, then $c\hbox{-}q\hbox{-}a\hbox{-}q''\hbox{-}d\hbox{-}s$ is and induced $P_6$. 
So $S(c)\cup S(d)\subseteq S(a)$, which is disjoint from $S(b)$, and therefore $b\sim c$ and $b\sim d$. For a point $r_b\in S(b)$, the path $r_b\hbox{-}b\hbox{-}c\hbox{-}q\hbox{-}a\hbox{-}q''$ is an induced $P_6$.

So $a\sim d$. Let $p$ be a private vertex of $q$. By~\eqref{eq:stab}, $p\not\sim a$ and $p\not\sim c$. Choose $r_a\in S(a)\setminus\{q\}$, $r_c\in S(c)\setminus\{q\}$, $r_b\in S(b)$, $r_d\in S(d)$. If $p\sim b$, then $r_b\hbox{-}b\hbox{-}p\hbox{-}q\hbox{-}a\hbox{-}r_a$ is an induced $P_6$. If $p\sim d$, then $r_d\hbox{-}d\hbox{-}p\hbox{-}q\hbox{-}c\hbox{-}r_c$ is an induced $P_6$. So $p\not\sim b,d$. If $b\sim d$, then $p\hbox{-}q\hbox{-}a\hbox{-}d\hbox{-}b\hbox{-}r_b$ is an induced $P_6$. So $b$ and $d$ have a common point $q'$ and $p\hbox{-}q\hbox{-}a\hbox{-}d\hbox{-}q'\hbox{-}b$ is an induced $P_6$.
\end{proof}

\begin{lemma}\label{lem:maxt}
Suppose $G[Y]$ is connected and $\Phi$ is connected with at least one edge. Let $T$ be the set of $t\in Y\setminus V(\Phi)$ with $V(\Phi)\subseteq\Sh(t)$, and choose $t\in T$ with $|\Adj(t)\cap V(\Phi)|$ maximum. If $x,y\in\Cov(t)$ and $xy\in E(\Phi)$, then $N(x)\cap Y\subseteq N(t)$ and $N(y)\cap Y\subseteq N(t)$.
\end{lemma}

\begin{proof}
The set $T$ is nonempty by \cref{lem:uct}. Note that $t\not\sim x$ and $t\not\sim y$ by~\eqref{eq:stab}. We use the following observation.

\smallskip\noindent($\ast$) \emph{If $x',y'\in\Cov(t)$, $x'y'\in E(\Phi)$, and $w'\in N(x')\cap Y\setminus N(t)$, then $w'$ connects $x'y'$ and $x'\in\Adj(w')$.}

\smallskip\noindent Indeed, for $i\in\{1,2\}$, if $w'\not\sim y'$ and $r_{y'}^i\notin S(w')$, then $w'\hbox{-}x'\hbox{-}r_{x'}^1\hbox{-}t\hbox{-}r_{y'}^i\hbox{-}y'$ is an induced $P_6$. So $w'\sim y'$ or $r_{y'}^1,r_{y'}^2\in S(w')$. Since $w'\sim x'$ and $X_{x'},X_{y'}$ are anticomplete, $w'\notin X_{x'}\cup X_{y'}$. So $w'$ connects $x'y'$. Since $w'\sim x'$, \cref{lem:kinds} gives $x'\in\Adj(w')$.

\smallskip
Now suppose $w'\in N(x)\cap Y\setminus N(t)$. By ($\ast$) and Lemma \cref{lem:kinds}, $x,y\in\Sh(w')$ and $x\in\Adj(w')$. Let $C$ be the vertex set of the component of $\Phi[\Sh(w')\cap V(\Phi)]$ containing $x$. Suppose $C\ne V(\Phi)$. Since $\Phi$ is connected, some edge $mu$ of $\Phi$ has $m\in C$ and $u\notin C$, and then $u\notin\Sh(w')$. Part (i) of the claim in the proof of \cref{lem:uct}, applied with $t$ replaced by $w^\prime$, shows that $u$ is $\Phi$-adjacent to every vertex of $C$,
in particular to $x$. 
Since $X_u$ and $X_x$ are anticomplete and $w'\sim x$, $w'\notin X_u$.
Note that $w'$ has no neighbor in $X_u$, for otherwise $w'$ would connect $xu$ (it touches $X_x$ at $x$, and $w'\notin X_u$ as $w'\sim x$) and $u\in\Sh(w')$ by \ref{lem:shadow} (b). 
As $t\in T$, $u\in\Sh(t)$. If $u\in\Cov(t)$, then ($\ast$) applied to $x'=x$, $y'=u$ shows that $w'$ connects $xu$, a contradiction. If $u\in\Adj(t)$, then $r_u^1\hbox{-}u\hbox{-}t\hbox{-}r_x^1\hbox{-}x\hbox{-}w'$ is an induced $P_6$, a contradiction. Hence $V(\Phi)\subseteq\Sh(w')$. If $w'\in V(\Phi)$, it has a $\Phi$-neighbor $v\in\Sh(w')$ and \cref{lem:shadow} (a) gives $w'\notin X_{w'}$, which is a contradiction. So $w'\in T$.

Next let $z\in\Adj(t)\cap V(\Phi)$. Then $z\in\Sh(w')$. If $z\in\Cov(w')$, then $r_x^1\hbox{-}t\hbox{-}p_z^2\hbox{-}r_z^2\hbox{-}w'\hbox{-}r_z^1$ is an induced $P_6$ (here $r_z^2\ne r_x^1$ as $z\sim t$). So $z\in\Adj(w')$. Thus $\Adj(t)\cap V(\Phi)\subseteq\Adj(w')$ and $x\in\Adj(w')\setminus\Adj(t)$ as $t\nsim x$. This contradicts the choice of $t$.
\end{proof}

\begin{proof}[Proof of Theorem \ref{thm:stablestars}]

We use induction on $w=\omega(Y)$. If $w\le1$ then $Y$ is stable and the bound holds. Let $w\ge2$. Since $\chi(Y)$ is the maximum over the components of $G[Y]$, we may assume $G[Y]$ is connected, and we may replace $R$ by a minimal subset $R'$ with $(Y,R')$ a dominated pair; this keeps the stars stable. Let $Y_1=Y\setminus Y_{\ge2}$ be the set of vertices with one point. By \cref{thm:cg}, $\chi(Y_1)\le w^2$.

Let $U\subseteq V(\Phi)$ meet every edge of $\Phi$. Then $(Y_{\ge2}\setminus U,R)$ has no far pair. \cref{thm:nofar} (with $M\le1$) gives $\chi(Y_{\ge2}\setminus U)\le w^2+2w$. Hence
\begin{equation}\label{eq:split}
\chi(Y)\le 2w^2+2w+\chi(U).
\end{equation}
Since $5w^3-5(w-1)^3=15w^2-15w+5$ and $13w^2\ge17w$ for $w\ge2$, it suffices to find $U$ with $\chi(U)\le 5(w-1)^3+5$.

If $\Phi$ has no edge, take $U=\emptyset$.
So $\Phi$ has at least one edge.
Suppose first that $\Phi$ has $m\ge2$ components, with vertex sets $V_1,\dots,V_m$. 
Take $U=V(\Phi)$. Two edges in different components satisfy the hypothesis of \cref{lem:2k2}, and so $V_i$ is complete to $V_j$ for $i\ne j$. Hence $\chi(U)=\sum_i\chi(V_i)$ with $\omega_i=\omega(V_i)\ge1$, $\sum_i\omega_i\le w$ and $\omega_i\le w-1$. By induction $\chi(U)\le 5\sum_i\omega_i^3$. For integers $p\ge q\ge1$ we have $p^3+q^3\le (p+q-1)^3+1$, so $\sum_i\omega_i^3\le (w-m+1)^3+(m-1)\le(w-1)^3+1$.

Now suppose that $\Phi$ is connected. Choose $t$ as in \cref{lem:maxt}, let $W_t$ be the set of vertices $x\in\Cov(t)$ that have a $\Phi$-neighbor in $\Cov(t)$, and take $U=(\Adj(t)\cap V(\Phi))\cup W_t$. Every edge $ab$ of $\Phi$ has $a,b\in\Sh(t)$. So one of $a,b$ lies in $\Adj(t)$ or both lie in $W_t$. Hence, $U$ meets every edge. Since $\Adj(t)\subseteq N(t)$ and $\omega(N(t)\cap Y)\le w-1$, induction gives $\chi(\Adj(t)\cap V(\Phi))\le 5(w-1)^3$. By \cref{lem:maxt}, every $x\in W_t$ has $N(x)\cap Y\subseteq N(t)$, while $t$ is non-adjacent to every vertex of $W_t$. So $W_t$ is stable. Hence $\chi(U)\le 5(w-1)^3+1$. 
\end{proof}

\section{The overlap graph}\label{sec:overlap}

Let $(Y,R)$ be a dominated pair. Two vertices $y,y'\in Y$ \emph{overlap} if $y\sim y'$ and $S(y)\cap S(y')\ne\emptyset$. The \emph{overlap graph} $\Ov=\Ov(Y,R)$ has vertex set $Y$ and the overlapping pairs as edges. Note that $\Ov(Y',R)=\Ov(Y,R)[Y']$ for $Y'\subseteq Y$.

\begin{theorem}\label{thm:overlap}
Let $(Y,R)$ be a dominated pair with $Y\ne\emptyset$, $w=\omega(Y)$ and $M=M(Y,R)$. Then $\chi(\Ov(Y,R))\le 11w^7M$.
\end{theorem}

We first prove \cref{thm:stable} and \cref{cor:blob} from \cref{thm:overlap,thm:stablestars}.

\begin{proof}[Proof of \cref{thm:stable}]
Note that
$(Y,R)$ is a dominated pair. Color $\Ov(Y,R)$ properly with at most $11w^7M$ colors. If $C$ is a color class and $r\in R$, then $I_r\cap C$ is stable in $Y$, since two adjacent vertices of $I_r$ overlap. So $(C,R)$ is a dominated pair with stable stars, and $\chi(C)\le 5w^3$ by \cref{thm:stablestars}. Hence $\chi(Y)\le 55w^{10}M$.
\end{proof}

\begin{proof}[Proof of \cref{cor:blob}]
Let $B=N(A_1\cup\dots\cup A_k)$.
Since the $A_i$ are pairwise anticomplete, $N(A_i)\subseteq B$. Let $H$ be the graph of \cref{lem:contract} with $U=B$ so that $N_H(z_i)=N(A_i)$. Then $H$ is $P_6$-free. Since $(B,\{z_1,\dots,z_k\})$ is a dominated pair in $H$, \cref{thm:stable} gives $\chi(B)\le 55\,\omega(B)^{10}\max_i\chi(N(A_i))$.
\end{proof}

\subsection{The root tree}
For the rest of this section fix $(Y,R)$, $w$ and $M$ as in \cref{thm:overlap}. 
We use induction on $w$. 
An \emph{$\Ov$-component} of a set $Z\subseteq Y$ is the vertex set of a component of $\Ov[Z]$.

We build a rooted forest $\mathcal{T}$ whose \emph{nodes} are subsets of $Y$ with clique number $w$. The top nodes of $\mathcal{T}$ are the $\Ov$-components of $Y$ with clique number $w$. For a node $K$, choose a vertex $z_K\in K$ with $|S(z_K)|$ maximum among the vertices of $K$, called the \emph{root vertex} of $K$. Put
\[
T_K=S(z_K),\qquad Q_K=\{y\in K\setminus N[z_K]: S(y)\cap T_K\ne\emptyset\},\qquad B_K=\{z_K\}\cup Q_K.
\]
The children of $K$ are the $\Ov$-components of $K\setminus B_K$ with clique number $w$. The $\Ov$-components of $K\setminus B_K$ with smaller clique number and the $\Ov$-components of $Y$ with clique number less than $w$ are called \emph{small pieces}. Every vertex of $Y$ lies in exactly one set $B_K$ or in exactly one small piece. Let $U=\bigcup_{K\in \mathcal{T}} B_K$. A vertex $y\in Q_K$ is a \emph{twin} if $S(y)=T_K$, and a {\em non-twin} if $y\in Q_K$ and $S(y)\neq T_K$. 
When we say that a node $K'$ is \emph{below} a node $K$, we mean that $K'$ is a proper descendant of $K$.

\begin{lemma}\label{lem:tree}
The following holds.
\begin{enumerate}
\item No edge of $\Ov$ joins two different small pieces, and no edge of $\Ov$ joins $B_K$ to $B_{K'}$ unless $K=K'$ or one of $K,K'$ is below the other.
\item If $K'$ is below $K$ and $y\in K'$, then $|S(y)|\le|T_K|$, and either $y\sim z_K$ or $S(y)\cap T_K=\emptyset$.
\item If $K'$ is below $K$ and $T_K\cap T_{K'}\ne\emptyset$, then $z_K\sim z_{K'}$.
\end{enumerate}
\end{lemma}

\begin{proof}
(a) If two such sets are not nested in the forest, then they lie in different $\Ov$-components of $Y$ or there is a node $K''$ such that they lie in different $\Ov$-components of $K''\setminus B_{K''}$.
(b) Since $K'\subseteq K\setminus B_K$, $y\in K$ and $y\notin B_K$.
(c) follows from (b) with $y=z_{K'}$.
\end{proof}

\begin{lemma}\label{lem:rootadj}
Let $K'$ be below $K$ and $y\in Q_{K'}$ be a non-twin with $y\sim z_K$. Then $z_K\sim z_{K'}$.
\end{lemma}

\begin{proof}
Suppose that $z_K\not\sim z_{K'}$. By \cref{lem:tree}(c), $T_K\cap T_{K'}=\emptyset$ . 
Choose $a\in S(y)\cap T_{K'}$.
Suppose that $T_{K'}\subseteq S(y)$. Then $|T_{K'}|\le |S(y)|$. By the choice of $z_{K'}$, $|S(y)|\le |T_{K'}|$. So 
$S(y)=T_{K'}$ and this contradicts $y$ is a non-twin.
So there is a vertex $f\in T_{K'}\setminus S(y)$.
By \cref{lem:tree}(b), $|S(y)|\le |T_K|$.
Since $a\in S(y)\cap T_{K'}$ and $T_K\cap T_{K'}=\emptyset$, we have $S(y)\setminus T_K\ne\emptyset$.
Thus $T_K\setminus S(y)\ne\emptyset$ and so there is a vertex $c\in T_K\setminus S(y)$.
Then $f\hbox{-}z_{K'}\hbox{-}a\hbox{-}y\hbox{-}z_K\hbox{-}c$ is an induced $P_6$.
\end{proof}

\subsection{Coloring the root vertices}

For a root vertex $z=z_K$, let $\mathrm{anc}(z)$ be the set of root vertices of the nodes that $K$ is below. Let $h(z)$ be the largest $k$ for which there are $z_1,\dots,z_k\in\mathrm{anc}(z)$ with $S(z)\subsetneq S(z_1)\subsetneq\dots\subsetneq S(z_k)$, and let $t(z)$ be the number of $z'\in\mathrm{anc}(z)$ with $S(z')=S(z)$.

\begin{lemma}\label{lem:roots}
There is a map $c_{\mathrm{root}}$ from the set of root vertices to a set of $w^3$ colors such that $c_{\mathrm{root}}(z)\ne c_{\mathrm{root}}(z')$ whenever $z'\in\mathrm{anc}(z)$ and $z\sim z'$.
\end{lemma}

\begin{proof}
A chain as in the definition of $h(z)$, together with $z$, consists of root vertices of nested nodes whose types all contain $S(z)$. By \cref{lem:tree} (c), they are pairwise adjacent, and so $h(z)\le w-1$. Similarly, $t(z)\le w-1$.

Fix a maximal chain of nodes $K_0\supset K_1\supset\cdots$ of $\mathcal{T}$ (called a \emph{branch}) and a value $(h_0,t_0)$. 
Let $L$ be the set of root vertices on the branch with $(h,t)=(h_0,t_0)$. 
Then Two vertices of $L$ have incomparable types. Moreover $G[L]$ is $P_4$-free.
Suppose not and let $a\hbox{-}b\hbox{-}c\hbox{-}d$ be an induced $P_4$ in $L$. By \cref{lem:tree} (c), the non-adjacent pairs $ac$, $ad$, $bd$ have disjoint types. Choose $f\in S(a)\setminus S(b)$ and $g\in S(d)\setminus S(c)$. Then $f\hbox{-}a\hbox{-}b\hbox{-}c\hbox{-}d\hbox{-}g$ is an induced $P_6$.

Now process the root vertices in order of depth, and let $g(z)$ be the least integer $j\ge0$ such that no $z'\in\mathrm{anc}(z)$ with $(h(z'),t(z'))=(h(z),t(z))$ and $z'\sim z$ has $g(z')=j$. On each branch, $g$ restricted to $L$ is the greedy coloring of the $P_4$-free graph $G[L]$ in order of depth. So $g(z)\le\omega(L)-1\le w-1$ by \cref{lem:chvatal}. It follows that $c_{\mathrm{root}}(z)=(h(z),t(z),g(z))$ is the desired map.
\end{proof}

\subsection{Non-neighbors of a vertex with disjoint types}

\begin{lemma}\label{lem:far}
Let $z\in Y$, $T=S(z)$, $v,v'\in F(z)$ and $x\in Q(z)$ with $T\not\subseteq S(x)$.
\begin{enumerate}
\item If $x$ and $v$ have a common point, then $x\sim v$.
\item If $x\sim v$, then $S(v)\subseteq S(x)$.
\item If $x\sim v$ and $v\sim v'$, then $x\sim v'$.
\end{enumerate}
\end{lemma}

\begin{proof}
Fix $c\in T\setminus S(x)$ and $a\in S(x)\cap T$. For (a), if $p$ is a common point and $x\not\sim v$, then $c\hbox{-}z\hbox{-}a\hbox{-}x\hbox{-}p\hbox{-}v$ is an induced $P_6$ (note that $p\notin T$ as $S(v)\cap T=\emptyset$). 
For (b), if $q\in S(v)\setminus S(x)$, then $c\hbox{-}z\hbox{-}a\hbox{-}x\hbox{-}v\hbox{-}q$ is an induced $P_6$. 
For (c), if $x\not\sim v'$, then $c\hbox{-}z\hbox{-}a\hbox{-}x\hbox{-}v\hbox{-}v'$ is an induced $P_6$.
\end{proof}

\subsection{Proof of \texorpdfstring{\cref{thm:overlap}}{Theorem 5.1}}
For $v\in Q_K$ put $\kappa_1(v)=c_{\mathrm{root}}(z_K)$.

\begin{lemma}\label{lem:class}
Let $x\in Q_K$ and $u\in Q_{K'}$ be non-twins with $\kappa_1(x)=\kappa_1(u)$, where $K'$ is below $K$. Then $u\in F(z_K)$. Consequently, if $x,y,u$ are non-twins with the same value of $\kappa_1$ lying in nodes $K,K',K''$ with $K''$ below $K'$ and $K'$ below $K$ and $xy$ and $yu$ are edges of $\Ov$, then $xu$ is an edge of $\Ov$.
\end{lemma}

\begin{proof}
If $u\sim z_K$, then $z_K\sim z_{K'}$ by \cref{lem:rootadj}. By \cref{lem:roots}, $c_{\mathrm{root}}(z_K)\ne c_{\mathrm{root}}(z_{K'})$, a contradiction. So $u\not\sim z_K$. By \cref{lem:tree} (b), $S(u)\cap T_K=\emptyset$, that is, $u\in F(z_K)$. 
For the second statement, $y,u\in F(z_K)$ by the first statement. Since $x$ is a non-twin and $|S(x)|\le|T_K|$, we have $T_K\not\subseteq S(x)$. Applying \cref{lem:far} to $z=z_K$ and $x$, we have   $x \sim u$ (by (c)) and  $S(u)\subseteq S(x)$ (by (b)). So $x$ and $u$ overlap.
\end{proof}

Now we color $\Ov$. The small pieces are dominated pairs with clique number at most $w-1$ and $M(\cdot,R)\le M$. By inductive hypothesis, each has an overlap coloring with at most $11(w-1)^7M$ colors, and by \cref{lem:tree} (a) one palette serves for all of them. We color the rest vertices $U$ of $\Ov$ with a disjoint palette as follows.
\begin{itemize}
\item A root vertex $z$ gets the color $(\mathsf{root},c_{\mathrm{root}}(z))$.
\item For each node $K$ fix a proper coloring $\rho$ of $G[Q_K]$ with at most $(w^2+4w)M$ colors: it exists by \cref{cor:capt} applied to $(K,R)$ and $t=z_K$.
\item For each value $k$ of $\kappa_1$, let $\Gamma_k$ be the graph on the non-twins $v$ with $\kappa_1(v)=k$ whose edges are the edges of $\Ov$ between vertices in different nodes. By \cref{lem:tree} (a), the nodes of the ends of such an edge are nested. If we orient the edge towards the deeper node, this orientation is transitive by \cref{lem:class}. So $\Gamma_k$ is a comparability graph, and its cliques are cliques of $G$. Hence $\Gamma_k$ has a proper coloring $\kappa_2$ with at most $w$ colors.
\item A non-twin $v\in Q_K$ gets the color $(\mathsf{non},\rho(v),\kappa_1(v),\kappa_2(v))$, and a twin $v\in Q_K$ gets the color $(\mathsf{twin},\rho(v),\kappa_1(v))$.
\end{itemize}
We check that this coloring of $U$ is proper on $\Ov$. Edges between different kinds of vertices get different first coordinates. An edge inside one $Q_K$ is colored properly by $\rho$. By \cref{lem:tree} (a), every other edge joins two vertices in nested nodes $K\ne K'$. For two root vertices, \cref{lem:roots} applies. Two twins $y\in Q_K$, $u\in Q_{K'}$ that overlap share a point of $T_K\cap T_{K'}$. So $z_K\sim z_{K'}$ by \cref{lem:tree} (c) and $\kappa_1(y)\ne\kappa_1(u)$. Two non-twins with the same $\kappa_1$ are adjacent in some $\Gamma_k$ and are separated by $\kappa_2$.

The number of colors used on $U$ is at most
\[
w^3+(w^2+4w)M\cdot w^3\cdot w+(w^2+4w)M\cdot w^3\le w^3M+10w^6M\le 11w^6M,
\]
using $M\ge1$. So $\chi(\Ov)\le 11(w-1)^7M+11w^6M\le 11w^7M$. This proves \cref{thm:overlap}. \qed

\section{Proof of \texorpdfstring{\cref{thm:main}}{Theorem 1.1}}\label{sec:main}

\subsection{Tight graphs}
Throughout this section let
\[
d=45,\qquad \delta=\tfrac14,\qquad c=\frac{d}{\log(1/(1-\delta))}=\frac{45}{\log(4/3)},\qquad f(x)=x^{c\log x}\ (x\ge1).
\]
For a graph $G$ with $\omega(G)=W$, let $\nu=W^{-d}\chi(G)$. We say that $G$ is \emph{tight} if every $X\subseteq V(G)$ with $\chi(X)\ge\nu$ satisfies $\omega(X)>(1-\delta)W$. A set $X$ with $\chi(X)<\nu$ is \emph{negligible}.

\begin{lemma}\label{lem:tight}
If \cref{thm:main} is false, then some $P_6$-free tight graph $G$ has $\omega(G)\ge3$.
\end{lemma}

\begin{proof}
Let $G$ be a $P_6$-free graph with $\chi(G)>f(\omega(G))$ and $W=\omega(G)$ minimum. 
By \cref{lem:gyarfas}, $\chi\le5\le 2^c=f(2)$, we have $W\ge3$. Let $L=\log(1/(1-\delta))\le\log W$. Then
\[
\log f(W)-\log f((1-\delta)W)=c\bigl(\log^2W-(\log W-L)^2\bigr)=cL(2\log W-L)\ge cL\log W=d\log W.
\]
If $\omega(X)\le(1-\delta)W$, then by the minimality of $W$ and the monotonicity of $f$, $\chi(X)\le f((1-\delta)W)\le W^{-d}f(W)<W^{-d}\chi(G)$. So $G$ is tight.
\end{proof}

For the rest of this section let $G$ be a $P_6$-free tight graph with $\omega(G)=W\ge3$, and write $\chi=\chi(G)$. Then the following holds.
\begin{itemize}
\item Every $X$ with $\omega(X)\le(1-\delta)W$ is negligible. In particular, a single vertex is negligible and so $\chi>W^d$.
\item If $\chi(X)\ge\nu$, then $\omega(\CN(X))\le W-\omega(X)<\delta W$ and so $\CN(X)$ is negligible.
\item By \cref{cor:blob} and $55\,W^{10}\le W^{16}$, for pairwise anticomplete connected sets $A_1,\dots,A_k$,
\begin{equation}\label{eq:blob16}
\chi(N(A_1\cup\dots\cup A_k))\le W^{16}\max_i\chi(N(A_i)).
\end{equation}
Similarly, \cref{thm:stable} holds with the factor $W^{16}$ in every induced subgraph of $G$.
\end{itemize}

\subsection{Cutting off a piece}
A \emph{tail} of a vertex $v$ with respect to a set $X$ is an induced path $v\hbox{-}b\hbox{-}a$ with $a,b\notin X$ and $\{a,b\}$ anticomplete to $X$.

\begin{lemma}\label{lem:endpoint}
Let $X$ be connected and co-connected with $|X|\ge2$, and $v\notin X$ be mixed on $X$ and have a tail with respect to $X$. Then the maximal strong modules of $G[X]$ in which $v$ has a non-neighbor form a stable set of the quotient graph.
\end{lemma}

\begin{proof}
Let $v\hbox{-}b\hbox{-}a$ be a tail and let $D$ be a component of $G[X\setminus N(v)]$. If some $x\in X\cap N(v)$ is mixed on $D$, choose an edge $yz$ of $G[D]$ with $x\sim y$, $x\not\sim z$, and then $a\hbox{-}b\hbox{-}v\hbox{-}x\hbox{-}y\hbox{-}z$ is an induced $P_6$. So $D$ is a module of $G[X]$. Since $v$ has at least one neighbor in $X$, $D\ne X$. By \cref{lem:gallai} (c), $D$ lies in a maximal strong module. If $v$ had non-neighbors $u\in M_1$ and $u'\in M_2$ for adjacent maximal strong modules $M_1,M_2$, then $u\sim u'$. So $u,u'$ lie in one component of $G[X\setminus N(v)]$ and hence in one maximal strong module, a contradiction.
\end{proof}

\begin{lemma}[Cut-off lemma]\label{lem:cut}
Let $d_1\le d-3$, and $P$ be a connected set with $\chi(P)\ge W^{-d_1}\chi$. Then there is a strong module $X$ of $G[P]$ such that
\begin{enumerate}
\item $X$ is connected and co-connected, $\omega(X)>(1-\delta)W$, and $\chi(X)\ge\chi(P)-\tfrac12W^{2-d}\chi$;
\item there are cliques $K_1,\dots,K_m\subseteq X$ with $m\le W$ and $|K_i|>(\frac12-\delta)W$ such that every vertex $v\notin X$ that is mixed on $X$ and has a tail with respect to $X$ is complete to some $K_i$;
\item each $\CN(K_i)$ in (b) is negligible.
\end{enumerate}
\end{lemma}

\begin{proof}
We walk down the modular decomposition tree of $G[P]$. Put $X_0=P$. We then recursively define strong module $X_j$ of $P$ for $j\ge 1$ as follows.
Suppose that $X_j$ is a strong module of $G[P]$ with $\chi(X_j)\ge \frac12W^{-d_1}\chi$. Then $\chi(X_j)\ge\nu$ and so $\omega(X_j)>(1-\delta)W$. In particular, $|X_j|\ge 2$.
\begin{itemize}
\item If $X_j$ is not connected, let $X_{j+1}$ be a component of maximum chromatic number. In this case, $\chi(X_{j+1})=\chi(X_j)$.
\item If $X_j$ is connected but not co-connected, its co-components $D_1,\dots,D_s$ ($s\le W$) are pairwise complete. If there are two components with chromatic number $\ge \nu$, the clique number of the union of them is larger than $2(1-\delta)W>W$. 
As $\chi(X_j)\ge \frac12W^{-d_1}\chi\ge W^{1-d}\chi=W\cdot \nu$, exactly one co-component $X_{j+1}$ is not negligible.
Let $X_{j+1}$ be the unique co-component that is not negligible. In this case,
$\chi(X_{j+1})\ge\chi(X_j)-(W-1)\nu$. Also $\omega(X_{j+1})\le\omega(X_j)-1$.
\item If $X_j$ is connected and co-connected, let $\mathcal{H}$ be the set of its maximal strong modules with clique number larger than $W/2$, and $U_{\mathcal{H}}$ their union. Note that any two modules in $\mathcal{H}$ are anticomplete and so
$\chi(U_{\mathcal{H}})=\max_{M\in\mathcal{H}}\chi(M)$. If $\chi(X_j\setminus U_{\mathcal{H}})\ge\nu$, we stop. 
Otherwise let $X_{j+1}\in\mathcal{H}$ have maximum chromatic number. In this case,  $\chi(X_{j+1})>\chi(X_j)-\nu$. 
By Lemma \ref{lem:gallai} (c), $X_{j+1}$ is complete to some other module and so $\omega(X_{j+1})\le\omega(X_j)-1$.
\end{itemize}
In each step, the chromatic number of $X_j$ decreases at most by $(W-1)\nu$. 
Since $P$ is connected, the first step is not the first kind. 
Two consecutive steps are never of the first kind and so at least $\lceil j/2\rceil$ of the first $j$ steps lower the clique number, and $\omega(X_j)\le W-\lceil j/2\rceil$. 
Since $\omega(X_j)>(1-\delta)W$,
$\lceil j/2\rceil<\delta W$ and so $j<2\delta W$. 
Since $j+1\le 2\delta W+1\le W$ and $W^{d-d_1-2}\ge2$,
$\chi(X_{j+1})\ge\chi(P)-(j+1)(W-1)\nu\ge\chi(P)-W^{2-d
}\chi\ge\frac12W^{-d_1}\chi$.
So the invariant is kept. 
Since the walk cannot continue forever, it stops at some $X=X_J$ that is connected and co-connected with $\chi(X\setminus U_{\mathcal{H}})\ge\nu$, where $J<2\delta W$ and $\chi(X)\ge\chi(P)-J(W-1)\nu\ge\chi(P)-\frac12W^{2-d}\chi$. This proves~(a).

For (b), $\chi(X\setminus U_{\mathcal{H}})\ge\nu$ gives $\omega(X\setminus U_{\mathcal{H}})>(1-\delta)W$. So there is a clique $\Lambda$ of the quotient, avoiding $\mathcal{H}$, such that taking a maximum clique $K_a$ in each module $a\in\Lambda$, the union $K_\Lambda=\bigcup_{a\in\Lambda}K_a$ is a clique of size more than $(1-\delta)W$. For $a\in\Lambda$ put $K'_a=K_\Lambda\setminus K_a$. Then $|\Lambda|\le W$ and $|K'_a|>(1-\delta)W-W/2$. Let $v$ be as in (b). By \cref{lem:endpoint}, $v$ has a non-neighbor in at most one module of $\Lambda$ and so it is complete to $K'_a$ for some $a\in\Lambda$.
This proves (b).
Note that $\omega(\CN(K'_a))<(\frac12+\delta)W=(1-\delta)W$, so each $\CN(K_i)$ is negligible. This proves (c).
\end{proof}

\begin{lemma}\label{lem:sep}
Let $Z$ be connected, $P\subseteq Z$ be connected , $S\subseteq Z$ be disjoint with $P$ such that every vertex of $S$ has a neighbor in $P$.
Let $Q$ be a component of $G[Z\setminus S]$ that is disjoint with $P$ and anticomplete to $P$.
Suppose $\chi(Q)\ge W^{-d_1}\chi$ with $d_1\le d-3$ and $\omega(P)>(1-\delta)W$. Let $X$ and $K_1,\dots,K_m$ be given by \cref{lem:cut} applied to $Q$. Then
\[
N_Z(X)\subseteq\CN(X)\cup\CN(P)\cup\CN(K_1)\cup\dots\cup\CN(K_m).
\]
\end{lemma}

\begin{proof}
As $Q$ is a component of $G[Z\setminus S]$, $N_Z(X)\subseteq Q\cup S$. 
Since $X$ is a module of $Q$ by \cref{lem:cut}, $N_Q(X)\subseteq \CN(X)$. 
Let $s\in S\cap N(X)$ be a vertex that is not complete to $X$. If $s$ is complete to $P$, it lies in $\CN(P)$. Otherwise $s$ is mixed on the connected set $P$, and so there is an edge $pp'$ of $G[P]$ with $s\sim p$, $s\not\sim p'$. Then $s\hbox{-}p\hbox{-}p'$ is a tail with respect to $X$, and $s\in\CN(K_i)$ for some $i$ by \cref{lem:cut} (b). 
\end{proof}

\subsection{Pieces}
Let $d_1=42, p=W^{-d_1}\chi, \tau_1=W^{40}p=W^{-2}\chi,\tau_2=\tau_1+(W+2)\nu, M_0=(W+2+W^{16})\nu$.

A \emph{piece system} is a family $\{A_1,\dots,A_k\}$ ($k\ge0$) of pairwise disjoint, pairwise anticomplete connected sets with $\chi(A_i)\ge p/2$ and $\chi(N(A_i))\le M_0$ for all $i$. Since $p/2\ge\nu$, each piece has $\omega(A_i)>(1-\delta)W$. For a piece system, let $B=N(A_1\cup\dots\cup A_k)$. The components of $G\setminus(B\cup A_1\cup\dots\cup A_k)$ are its \emph{free components}.

\begin{lemma}[Peeling]\label{lem:peel}
Let $\{A_1,\dots,A_k\}$ be a piece system and $Z$ a free component with $\chi(Z)\ge\tau_1$. Then there is $X\subseteq Z$ such that $\{A_1,\dots,A_k,X\}$ is a piece system, and some component $Z'$ of $G[Z\setminus N_Z(X)]$ is a free component of the new piece system with $\chi(Z')>\chi(Z)-(W+2)\nu$.
\end{lemma}

\begin{proof}
We proceed as follows.

\emph{Step 1: an anticomplete pair.} We have $\chi(Z)\ge W^{-2}\chi>W^{d-2}>W\ge\omega(Z)$, so $G[Z]$ is not complete, and $\omega(Z)>(1-\delta)W\ge2$. By \cref{lem:nguyen}, $Z$ contains an anticomplete pair of connected sets with chromatic numbers at least $\omega(Z)^{-40}\chi(Z)\ge W^{-40}\tau_1=p$. Among all such pairs $(P,Q)$ in $Z$ choose one with $\chi(P)+\chi(Q)$ maximum, then $|P|+|Q|$ maximum, and with $\chi(P)\ge\chi(Q)$. Let $S$ be an inclusion-minimal subset of $Z\setminus(P\cup Q)$ separating $P$ from $Q$ in $G[Z]$. The components of $G[Z\setminus S]$ containing $P$ and $Q$ form a pair at least as good as $(P,Q)$. So $P$ and $Q$ are components of $G[Z\setminus S]$. By minimality of $S$, every vertex of $S$ has a neighbor in $P$ and in $Q$.

\emph{Step 2: the new piece.} Apply \cref{lem:cut} to $Q$ (note $d_1=d-3$) and let $X$, $K_1,\dots,K_m$ be the outcome. Then $X$ is connected and $\chi(X)\ge p-\frac12W^{2-d}\chi\ge p/2$. As $\chi(P)\ge p\ge\nu$, $\omega(P)>(1-\delta)W$. By  \cref{lem:sep}, $N_Z(X)\subseteq\CN(X)\cup\CN(P)\cup\CN(K_1)\cup\dots\cup\CN(K_m)$.

\emph{Step 3: the boundary outside $Z$.} Since $N(Z)\subseteq B$, we have $N(X)=N_Z(X)\cup(B\cap N(X))$. Let $v\in B\cap N(X)$. If $v$ is complete to $X$, then $v\in\CN(X)$. If $v$ is mixed on $X$ and has a tail with respect to $X$, then $v\in\CN(K_i)$ for some $i$ by \cref{lem:cut} (b). 
Suppose that $v$ is mixed on $X$ and has no tail with respect to $X$. Each piece $A_j$ and the set $P$ are connected and anticomplete to $X$. If $v$ is mixed on one of them, an edge $cc'$ inside it with $v\sim c$, $v\not\sim c'$ would give a tail $v\hbox{-}c\hbox{-}c'$. So $v$ is complete or anticomplete to each $A_j$ and to $P$. Since $v\in B$, it is complete to some piece. Let $U$ be the set of such vertices $v$ that are anticomplete to $P$.

\emph{Step 4: the cost of $U$.} Choose $a_j\in A_j$ for each $j$ and let $R=\{a_1,\dots,a_k\}$. Then $R$ is stable, every vertex of $U$ has a neighbor in $R$, and $N(a_j)\cap U\subseteq\CN(A_j)$, which is negligible. By \cref{thm:stable} (with factor $W^{16}$) applied to $G[U\cup R]$, $\chi(U)<W^{16}\nu$.

\emph{Step 5: the new piece system.} We have
\[
N(X)\subseteq\CN(X)\cup\CN(P)\cup\CN(K_1)\cup\dots\cup\CN(K_m)\cup U.
\]
Since $\CN(X),\CN(P),\CN(K_1),\dots,\CN(K_m)$ are negligible, $\chi(N(X))<(W+2)\nu+W^{16}\nu=M_0$. As $X\subseteq Z$ is connected and anticomplete to every $A_j$, $\{A_1,\dots,A_k,X\}$ is a piece system.

\emph{Step 6: the surviving component.} Since $X\subseteq Q$ and $Q$ is a component of $G[Z\setminus S]$, $N_Z(X)\subseteq Q\cup S$ and so $P\cap N_Z(X)=\emptyset$.
Let $Z'$ be the component of $G[Z\setminus N_Z(X)]$ containing $P$. Every other component $C\ne X$ of $G[Z\setminus N_Z(X)]$ is connected and anticomplete to $P$; If $\chi(C)\ge p$ then $\chi(C)\le\chi(Q)$ by the choice of $(P,Q)$. If $\chi(C)< p$, then $\chi(C)<p\le\chi(Q)$. Also $\chi(X)\le\chi(Q)\le\chi(P)\le\chi(Z')$. Hence
\[
\chi(Z')=\chi(Z\setminus N_Z(X))\ge\chi(Z)-\chi(N_Z(X))>\chi(Z)-(W+2)\nu.
\]
Since $N(Z')\subseteq N_Z(X)\cup B\subseteq B\cup N(X)$ and $Z'$ is anticomplete to $X$ and to every $A_j$, $Z'$ is a free component of the new system.
\end{proof}

\begin{lemma}\label{lem:dissolve}
If a piece system has a piece $A^*$ with $\chi(A^*)\ge\tau_2$, then there is a piece system with more pieces.
\end{lemma}

\begin{proof}
Removing $A^*$ leaves a piece system, and $A^*$, being anticomplete to the other pieces, lies in one of its free components $Z$ with $\chi(Z)\ge\tau_2\ge\tau_1$. By \cref{lem:peel}, we obtain a new piece $X_1$ and a free component $Z'$ with $\chi(Z')>\tau_2-(W+2)\nu=\tau_1$. Applying \cref{lem:peel} again to $Z'$ gives a second new piece $X_2$. The resulting system has one more piece than the original one.
\end{proof}

\subsection{Conclusion}
Take a piece system $\{A_1,\dots,A_k\}$ with $k$ maximum (the empty family is a piece system, and $k\le|V(G)|$). By \cref{lem:peel}, every free component has $\chi<\tau_1$, and by \cref{lem:dissolve}, every piece has $\chi<\tau_2$. The pieces and the free components are pairwise anticomplete and so $\chi(V(G)\setminus B)<\tau_2$. By~\eqref{eq:blob16}, $\chi(B)\le W^{16}M_0$. Since $W\ge3$ gives $W+2\le W^{16}$,
\[
\chi\le\chi(B)+\chi(V(G)\setminus B)<\bigl(2W^{32-45}+W^{-2}+2W^{1-45}\bigr)\chi\le\Bigl(\frac{2}{3^{13}}+\frac19+\frac{2}{3^{44}}\Bigr)\chi<\chi,
\]
a contradiction. \qed

\section*{Declaration of generative AI use} 
The proof presented in this paper was developed through an interactive dialogue with Claude Opus 5.5 (Anthropic) and ChatGPT-6 (OpenAI). The authors used them to explore multiple possible proof strategies and directions, and the AI generated the initial draft under the authors' guidance. The authors then critically examined the generated proof, verified each step for correctness and rigor, and revised it where necessary. The authors take full responsibility for the final content.


\begin{thebibliography}{99}

\bibitem{chvatal1984}
V.~Chv\'atal.
\newblock Perfectly ordered graphs.
\newblock In {\em Topics on Perfect Graphs}, volume~21 of {\em Ann. Discrete Math.}, pages 63--68. North-Holland, Amsterdam, 1984.

\bibitem{gallai1967}
T.~Gallai.
\newblock Transitiv orientierbare {G}raphen.
\newblock {\em Acta Math. Acad. Sci. Hungar.}, 18:25--66, 1967.

\bibitem{ghm2003}
S.~Gravier, C.~T. Ho\`ang, and F.~Maffray.
\newblock Coloring the hypergraph of maximal cliques of a graph with no long path.
\newblock {\em Discrete Math.}, 272(2-3):285--290, 2003.

\bibitem{gyarfas1975}
A.~Gy\'arf\'as.
\newblock On {R}amsey covering-numbers.
\newblock In {\em Infinite and Finite Sets}, volume~10 of {\em Colloq. Math. Soc. J\'anos Bolyai}, pages 801--816. North-Holland, Amsterdam, 1975.

\bibitem{gyarfas1987}
A.~Gy\'arf\'as.
\newblock Problems from the world surrounding perfect graphs.
\newblock {\em Zastos. Mat.}, 19(3-4):413--441, 1987.



\bibitem{nguyen2025pairs}
T.~H. Nguyen.
\newblock On polynomially high-chromatic pure pairs.
\newblock \href{https://arxiv.org/abs/2504.21127v3}{\tt arXiv:2504.21127v3}.

\bibitem{nguyen2025p5}
T.~H. Nguyen.
\newblock Polynomial $\chi$-boundedness for excluding {$P_5$}.
\newblock \href{https://arxiv.org/abs/2512.24907}{\tt arXiv:2512.24907}. Preliminary version to appear in FOCS 2026.



\bibitem{oum2026}
S.~Oum.
\newblock Coloring graphs with no long induced path.
\newblock \href{https://arxiv.org/abs/2609.08847}{\tt arXiv:2609.08847}, 2026.

\bibitem{ss2020}
A.~Scott and P.~Seymour.
\newblock A survey of $\chi$-boundedness.
\newblock {\em J. Graph Theory}, 95(3):473--504, 2020.

\bibitem{sss2023}
A.~Scott, P.~Seymour, and S.~Spirkl.
\newblock Polynomial bounds for chromatic number. {IV}. {A} near-polynomial bound for excluding the five-vertex path.
\newblock {\em Combinatorica}, 43(5):845--852, 2023.

\bibitem{seinsche1974}
D.~Seinsche.
\newblock On a property of the class of $n$-colorable graphs.
\newblock {\em J. Combin. Theory Ser. B}, 16:191--193, 1974.



\end{thebibliography}
\end{document}